\documentclass[12pt]{amsart}
\usepackage{amsmath,amsthm,amsmath,amsrefs}

\usepackage{amsbsy,amsmath,amsfonts,amsthm,amssymb,color}

\theoremstyle{plain}
\newtheorem{mythm}{Theorem}[section]

\newtheorem{crol}{Corollary}[section]

\usepackage[pdftex]{graphicx}

\usepackage{mathrsfs}

\newcommand{\bs}{\backslash}

\newcommand{\N}{\mathbb{N}}

\newcommand{\Z}{\mathbb{Z}}

\newdimen\Squaresize \Squaresize=14pt
\newdimen\Thickness \Thickness=0.4pt

\def\Square#1{\hbox{\vrule width \Thickness
   \vbox to \Squaresize{\hrule height \Thickness\vss
      \hbox to \Squaresize{\hss#1\hss}
   \vss\hrule height\Thickness}
\unskip\vrule width \Thickness} \kern-\Thickness}

\def\Vsquare#1{\vbox{\Square{$#1$}}\kern-\Thickness}

\title[The radical of skew polynomial extensions of PI rings] {On the Jacobson radical of skew polynomial extensions of rings satisfying a polynomial identity}

\author{Blake W. Madill}
\thanks{The author thanks NSERC for its generous support.}
\address{University of Waterloo \\
Department of Pure Mathematics \\
Waterloo, Ontario \\
Canada  N2L 3G1\\}
\email{bmadill@uwaterloo.ca}



\begin{document}

\begin{abstract} Let $R$ be a ring satisfying a polynomial identity and let $D$ be a derivation of $R$. We consider the Jacobson radical of the skew polynomial ring $R[x;D]$ with coefficients in $R$ and with respect to $D$, and show that $J(R[x;D])\cap R$ is a nil $D$-ideal.  This extends a result of Ferrero, Kishimoto, and Motose, who proved this in the case when $R$ is commutative.   \end{abstract}


\maketitle

\section{Introduction}

Let $R$ be a ring (not necessarily unital).  We say that a map $D:R\rightarrow R$ is a derivation if it satisfies $D(a+b)=D(a)+D(b)$ and $D(ab)=D(a)b+aD(b)$ for all $a,b\in R$. Given a derivation $D$ of a ring $R$, we can construct the \emph{differential polynomial ring}, also known as a \emph{skew polynomial ring of derivation type,} $R[x;D]$, which as a set is given by
$$
\lbrace x^na_n+\cdots+xa_1+a_0: n\geq 0, a_i\in R\rbrace,
$$
with regular polynomial addition and multiplication given by $xa=ax+D(a)$ for $a\in R$ and then extending via associativity and linearity.

When $D$ is the trivial derivation of $R$ (i.e. $D(a)=0$ for all $a\in R$) then $R[x;D]$ is simply $R[x]$. In this case, Amitsur \cite{Amitsur} showed that the Jacobson radical $J(R[x])$ is the polynomial ring $(J(R[x])\cap R)[x]$ and that $J(R[x])\cap R$ is a nil ideal of $R$. Extending this result, Ferrero, Kishimoto, and Motose \cite{Ferrero} showed that $J(R[x;D])=(J(R[x;D])\cap R)[x;D]$ and if it is further assumed that $R$ is commutative then $J(R[x;D])\cap R$ is a nil ideal of $R$. We extend this result to polynomial identity rings, of which commutative rings are a special case.  Using the work of Ferrero, Kishimoto, and Motose \cite{Ferrero} then allows us to completely describe the Jacobson radical of the skew polynomial ring.

\begin{mythm}\label{mainresult}
Let $R$ be a polynomial identity ring and let $D$ be a derivation of $R$.  Then $S:=J(R[x;D])\cap R$ is a nil $D$-ideal of $R$ and $J(R[x;D])=S[x;D]$. 
\end{mythm}

As a direct consequence of our main result, we also get a result of Tsai, Wu, and Chuang \cite{Tsai}, which says that if $R$ is a PI ring with zero upper nilradical then $R[x;D]$ is semiprimitive. Namely, if we additionally make the assumption that $N(R)=(0)$ then we get that $S=(0)$ and so $J(R[x;D])=(0)$.  We note that in general one cannot deduce the work of Tsai, Wu, and Chuang \cite{Tsai} from our result because the nilradical of a ring is not in general closed under the action of a derivation. (For an example, see \cite{Bell}.) Thus one cannot hope to obtain our result as a consequence by passing to a semiprimitive homomorphic image.  

We note that other work on this topic has been done by other authors \cites{Jordan, Bergen, Bergen2, Isfahani}. In fact, a similar result to Theorem \ref{mainresult} can be found in \cite{Bergen2} in a more general context.

\begin{mythm}[\cite{Bergen2}, Corollary 3.5.]\label{BergenResult}
Let $L$ be a Lie algebra over a field $K$ which acts as $K$-derivations on a $K$-algebra $R$. This action determines a crossed product $R\ast U(L)$ where $U(L)$ is the enveloping algebra of $L$. Assume that $L\neq 0$ and that $R$ is either right Noetherian, a PI algebra, or a ring with no nilpotent elements. Then $J(R\ast U(L))=N\ast U(L)$ where $N$ is the largest $L$-invariant nil ideal of $R$. Furthermore $J(R\ast U(L))$ is nil in this case.\end{mythm}

While the above result is in a way more general than our result, we do not require that $R$ be a PI algebra over a field. Our result concerns PI rings, of which PI algebras are a special case. 

The outline of this paper is as follows.  In \S2, we give some background on polynomial identity rings that we will use in our proof.  Then in \S3 we give the proof of Theorem \ref{mainresult}.
\section{Preliminaries}

To prove our main result, we make use of some theory of polynomial identity rings, or PI rings for short. For in-depth introductions to this theory we refer the reader to \cite{RowenPI} and \cite{Drensky}. Let $\Z\langle x_1,\dots, x_d\rangle$ denote the free associative algebra over $\Z$ in $d$ variables. We say that $f(x_1,\dots, x_d)$ is a \emph{polynomial identity} (PI) for $R$ if $f(r_1,\dots, r_d)=0$ for all $r_1,\dots, r_d\in R$. We say that $R$ is a \emph{PI ring} if there exists a nonzero $f\in \Z\langle x_1,\dots, x_d\rangle$, for some $d\in \N$, such that $f$ is a polynomial identity for $R$ and at least one coefficient of $f$ is $1$. If $R$ is a PI ring it can be shown that there exists a polynomial identity for $R$ of the form
$$
x_1x_2\cdots x_d=\sum_{\sigma\in S_d\bs\lbrace {\rm id}\rbrace} c_\sigma x_{\sigma(1)}x_{\sigma(2)}\cdots x_{\sigma(d)},
$$
where each $c_\sigma\in\Z$. Examples of PI rings include commutative rings, finite-dimensional algebras over a field, and the quantum plane $\mathbb{C}_q\{x,y\}$ with relation $xy=qyx$ and $q$ a root of unity. We make use of several fundamental results from ring theory and polynomial identity theory, which can be found in \cite{RowenPI}. In particular, if $R$ is a ring then $R$ has a a unique maximal nil ideal, denoted by $N(R)$ (\cite{Rowen PI}, Proposition 1.6.9.). Furthermore, if we assume that $R$ is a PI ring then $N(R)=\sum\lbrace \text{nil left ideals of }R\rbrace$ (\cite{RowenPI}, Corollary 1.6.18.).

The nil ideal $N(R)$ is called the \emph{upper nilradical} of $R$, or sometimes just the nilradical of $R$. We also get the following result about the centre of a polynomial identity ring, due to Rowen following on work by Posner, which will later allow us to use techniques found in \cite{Ferrero} for commutative rings. 

\begin{mythm}[\cite{RowenPI}, Theorem 1.6.14.]\label{centrethm} Let $R$ be a PI ring such that $N(R)=(0)$. Then every nonzero ideal of $R$ intersects the centre of $R$ nontrivially.
\end{mythm}

\section{Main Result}

In this section, we prove Theorem \ref{mainresult}---we note that we make use of some ideas from \cite{Ferrero} in this proof. 

\begin{proof}[Proof of Theorem \ref{mainresult}]
It is clear that $S$ is an ideal of $R$. Now let $a\in S$. Then $xa,ax\in J(R[x;D])$ and $xa=ax+D(a)$. Therefore $D(a)\in J(R[x;D])$ and so $D(a)\in S$. This shows that $S$ is a $D$-ideal.

It is left to show that $S$ is nil. To do this we consider the ring $L:=(S+N)/N$, where $N=N(R)$. We begin by noting that $L$ is an ideal of $R/N$, which is a polynomial identity ring with zero upper nilradical.   

If $L$ is the zero ideal of $R/N$ we have that $S\subseteq N$ and so $S$ is nil. Suppose, towards a contradiction, that $L$ is not the zero ideal. Thus by Theorem \ref{centrethm} we have that $L$ intersects the centre of $R/N$ nontrivially. Let $a\in S$ such that $a+N\neq N$ and $a+N$ is in the centre of $R/N$. We have that $xa\in J(R[x;D])$ and so by quasi-regularity there exists $f(x)\in J(R[x;D])$ such that 
\begin{align}
f(x)+xa-f(x)xa=0,
\end{align}
and
\begin{align}
f(x)+xa-xaf(x)=0.
\end{align}
Write
$$
f(x)=\sum_{i=0}^n x^ib_i.
$$
From (2) we immediately get that $b_0=0$. Now, from (1) we see that 
\begin{align*}
\sum_{i=0}^n x^ib_i+xa&-\sum_{i=0}^n \left(x^ib_ixa\right)=\sum_{i=0}^n x^ib_i+xa-\sum_{i=0}^n \left(x^i(xb_i-D(b_i))a\right)\\
&=\sum_{i=0}^n x^ib_i+xa-\sum_{i=0}^n x^{i+1}b_ia+\sum_{i=1}^n x^iD(b_i)a=0.
\end{align*}
Upon equating coefficients we get that:
\begin{align}
&b_na=0;\\
&b_i- b_{i-1}a+D(b_i)a=0, \text{ for } i=2,\ldots , n;\\
&b_1+a+D(b_1)a=0.
\end{align}
We have $J(R[x;D])=S[x;D]$ (see \cite{Ferrero})  and so $f(x)\in S[x;D]$. Therefore each $b_i\in S$.

We claim that $b_{n-j+1}a^{j}\in N$ for $j=1,\ldots ,n$. We prove this by induction on $j$. If $j=1$ we clearly have that $0=b_na\in N$. Now assume the result is true for $j<k$ with $k\in \{2,\ldots , n\}$ and consider the case when $j=k$. In particular, we have that $b_{n-k+2}a^{k-1}\in N$. Now, post-multiplying (4), taking $i={n-k+2}$, by $a^{k-1}$ we have by the induction hypothesis that
\begin{equation}
\begin{split}
b_{n-k+2}a^{k-1}&-b_{n-k+1}a^k+D(b_{n-k+2})a^k\\ &\equiv -b_{n-k+1}a^k+D(b_{n-k+2})a^k ~(\bmod ~N)\\
&\equiv 0 ~(\bmod ~N).
\end{split}
\end{equation}
To prove the claim it is left to show that $D(b_{n-k+2})a^k\in N$. From (1) and (2) we also get that $xaf(x)=f(x)xa$. Recall that we have the identity
$$
ax^r=\sum_{\ell=0}^r (-1)^\ell{r \choose \ell} x^{r-\ell}D^{\ell}(a),
$$
for any $r\in\N$. But then we have that
\begin{equation}\begin{split}
0&=f(x)xa-xaf(x)=\sum_{i=0}^n x^{i+1}b_ia+\sum_{i=1}^n x^iD(b_i)a-\sum_{i=0}^n xax^ib_i\\
&=\sum_{i=0}^n x^{i+1}b_ia+\sum_{i=1}^n x^iD(b_i)a-\sum_{i=0}^n\sum_{\ell=0}^i(-1)^\ell {i\choose\ell}x^{i-\ell+1}D^\ell(a)b_i.
\end{split}\end{equation}
Equating coefficients of $x^{n-k+2}$ in (7) we have that
\begin{equation}\begin{split}
b_{n-k+1}a&+D(b_{n-k+2})a+(-1)^k {n\choose k-1}D^{k-1}(a)b_n\\&+(-1)^{k-1}{n-1\choose k-2}D^{k-2}(a)b_{n-1}\\&+ \dots+ {n-k+2\choose 1}D(a)b_{n-k+2}-ab_{n-k+1}=0.
\end{split}\end{equation}
Notice that since $a$ is central mod $N$, the first and last term of the left-hand-side of $(8)$ are identical mod $N$. Also, by the induction hypothesis $b_na^{k-1}, b_{n-1}a^{k-1},\dots,\\ b_{n-k+2}a^{k-1}\in N$. Therefore by post-multiplying (8) by $a^{k-1}$ and considering this equation modulo $N$ we have
$$
D(b_{n-k+2})a^k\in N.
$$
Using this and (6) we get that $b_{n-k+1}a^k\in N$. The claim is then true by induction. 

By the above claim with $j=n$ we have that $b_1a^n\in N$ and so by (5)
\begin{align}
0\equiv b_1a^{n}+a^{n+1}+D(b_1)a^{n+1}\equiv a^{n+1}+D(b_1)a^{n+1} ~(\bmod ~N).
\end{align}
To see that $D(b_1)a^{n+1}\in N$ we equate coefficients of $x$ in (7). From doing so we get that
$$
D(b_1)a+(-1)^{n+1}D^n(a)b_n+(-1)^{n}D^{n-1}(a)b_{n-1}+\dots+D(a)b_1=0.
$$
Post-multiplying this by $a^{n+1}$ and using our claim we observe that $D(b_1)a^{n+1}\in N$. Combining this with (9) we get that $a^{n+1}\in N$ and so $a$ is nilpotent. But then, using the fact that $a$ is central mod $N$, we see that $(a\Z+aR+N)/N$ is a nil ideal of $R/N$.  Since $N$ is a nil ideal of $R$, we see that $a\Z+aR+N$ is also a nil ideal of $R$.  But $N$ is the upper nilradical and so we see that $a\in N$ and so we have that $a+N=N$, which is a contradiction. We conclude that $L=(0)$ and so $S$ is nil.
\end{proof}

We get the following immediate corollary, which is part of the main result of \cite{Tsai}.

\begin{crol}
Let $R$ be a PI ring with $N(R)=(0)$. Then $R[x;D]$ is semiprimitive for any derivation $D$ of $R$. 
\end{crol}

\section{Acknowledgements}
The author thanks Jason Bell and the referee for many helpful comments and suggestions.

\end{document}